\documentclass{amsart}

\usepackage{amsmath}
\usepackage{amsfonts}
\usepackage{amssymb}
\usepackage{amsthm}
\usepackage{hyperref}
\usepackage{color}
\theoremstyle{plain}
\newtheorem{lem}{Lemma}[]
\newtheorem{cor}[lem]{Corollary}
\newtheorem{prop}[lem]{Proposition}
\newtheorem{thm}[lem]{Theorem}
\theoremstyle{definition}
\newtheorem{ex}[lem]{Example}
\newtheorem{disc}[lem]{Remark}
\renewcommand{\leq}{\leqslant}

\newcommand{\pd}{\operatorname{pd}}	
\newcommand{\depth}{\operatorname{depth}}	
\newcommand{\len}{\operatorname{length}}
\newcommand{\HH}{\operatorname{H}}
\newcommand{\spec}{\operatorname{Spec}}
\newcommand{\ideal}[1]{\mathfrak{#1}}
\newcommand{\fm}{\ideal{m}}
 \newcommand{\fp}{\ideal{p}}
\newcommand{\supp}{\operatorname{Supp}}
\newcommand{\bbz}{\mathbb{Z}}
\newcommand{\xra}{\xrightarrow}
\newcommand{\xla}{\xleftarrow}

\begin{document}

\bibliographystyle{amsplain}

\title[Peskine and Szpiro's Dimension inequality]{Peskine and Szpiro's
  Dimension inequality holds\\ over local rings of depth at most one}

\author{Lars Winther Christensen}
\address{Texas Tech University\\
  Lubbock, TX 79409, U.S.A.}  \email{lars.w.christensen@ttu.edu}
\urladdr{https://larswinther.github.io/homepage/}

\author{Saeed Nasseh}
\address{Department of Mathematical Sciences\\
  Georgia Southern University\\
  Statesboro, GA 30460, U.S.A.}  \email{snasseh@georgiasouthern.edu}
\urladdr{https://cosm.georgiasouthern.edu/math/saeed.nasseh}

\thanks{L.W.C.\ was partly supported by Simons Foundation
  collaboration grant 962956}




\begin{abstract}
  We prove that the general version---the one that includes the
  dimension of the tensor product---of Peskine and Szpiro's Dimension
  Inequality holds over noetherian local rings of depth at most one.
\end{abstract}

\maketitle

\noindent Let $R$ be a commutative noetherian local ring. Serre
\cite{Serre} proved that if $R$ is regular and $M$ and $N$ are
finitely generated $R$-modules, then the inequality
$$
\dim_R M + \dim_R N \leq \dim R + \dim_R(M \otimes_R N)
$$
holds. It was conjectured by Peskine and Szpiro \cite{Peskin1} that
this inequality also would hold over a singular local ring, as long as
one of the modules, say $M$, has finite projective dimension. In this
paper, we will refer to this inequality as the \emph{Dimension
  Inequality}, though some authors reserve that name for the special
case where the $R$-module $M \otimes_R N$ has finite length and hence
dimension $0$.

An example  by Ma~\cite[\S2.3]{Ma} shows that the Dimension
Inequality does not hold in general. To be precise, Ma constructs a
Cohen--Macaulay local ring of dimension $3$ over which the
inequality fails. In that light, one asks under what extra assumptions
on the ring the inequality is valid. Our contribution is the following.

\begin{thm}
  \label{thm:main}
  Let $R$ be a commutative noetherian local ring and $M$ and $N$ be
  finitely generated $R$-modules. If $\pd_R{M}$ is finite and
  $\depth R \leqslant 1$ holds, then one has
  \begin{equation}
    \tag{1}
    \dim_R M + \dim_R N \leq \dim R + \dim_R(M \otimes_R N)\:.
  \end{equation}
\end{thm}

\begin{proof}
  By Peskine and Szpiro's Intersection Theorem, the inequality
  \begin{equation}
    \tag{2}
    \dim_RN \leq \pd_RM + \dim_R(M \otimes_R N)
  \end{equation}
  holds. It follows from the Auslander--Buchsbaum Formula and the
  assumptions on $M$ and $R$ that $\pd_RM$ is $0$ or $1$. In the
  former case, $M$ is free, so $\dim_R M = \dim R$ holds and thus
  $(1)$ follows from $(2)$ by adding this quantity to both side.

  Now assume that $\pd_RM = 1$ holds and consider the minimal free
  resolution
  \begin{equation*}
    0 \longrightarrow R^{a} \longrightarrow R^{b} \longrightarrow M \longrightarrow 0
  \end{equation*}
  where $a \leq b$ are positive integers. For every minimal prime
  ideal $\fp$ in $R$ it induces an exact sequence
  \begin{equation}
    \tag{3}
    0 \longrightarrow R_{\fp}^{a} \longrightarrow R_{\fp}^{b} \longrightarrow M_{\fp} \to 0
  \end{equation}
  of finite length $R_{\fp}$-modules.

  If $a = b$ holds, then $(3)$ yields
  $\len_{R_{\fp}}(M_{\fp})=0$ and thus $M_{\fp}=0$ for every minimal
  prime ideal $\fp$ of $R$. It follows that
  $\dim_RM \leq \dim R-1$ holds, and adding this to $(2)$ yields:
  \begin{align*}
    \dim_RM + \dim_RN & \leq \dim R-1+\pd_RM+\dim_R(M\otimes_RN)\\
                      &= \dim R + \dim_R(M \otimes_R N) \:.
  \end{align*}
  If $a<b$ holds, then $(3)$ yields
  $$\len_{R_{\fp}}(M_{\fp}) = (b-a)\len_{R_{\fp}}(R_{\fp})>0$$ and,
  therefore, $M_{\fp}\neq 0$ for every minimal prime ideal $\fp$ in
  $R$. Thus, $\supp_R M$ is all of $\spec R$, whence one has
  $$\supp_R(M\otimes_R N) = \supp_R M \cap \supp_R N = \supp_R N.$$ In
  particular, the equalities $\dim_R M =\dim R$ and
  $\dim_R N = \dim_R (M \otimes_R N)$ hold, whence equality holds in
  $(1)$.
\end{proof}

\begin{disc}\label{disc20260929b}
  The assumption on the projective dimension of $M$ in
  Theorem~\ref{thm:main} can not be dropped; see \cite[Exa.\
  18.5.4]{dcmca}. Also, the proof shows that, the assumption
  $\depth R\leq 1$ can be replaced by the weaker assumption
  $\pd_R M\leq 1$.
\end{disc}

Theorem~\ref{thm:main} extends to complexes as follows:

\begin{cor}\label{cor20260929a}
  Let $R$ be a commutative noetherian local ring and $M$ and $N$ be
  $R$-complexes with finitely generated homology modules and
  $\HH_i(M) = 0 = \HH_i(N)$ for $i \ll 0$. If $\pd_R (\HH_i(M))$ is
  finite for all $i \in \bbz$ and $\depth R \leq 1$ holds, then one has
  \begin{equation*}
    \dim_R M + \dim_R N \leq \dim R + \dim_R(M \otimes_R^{\bf L} N)\:.
  \end{equation*}
\end{cor}

\begin{proof}
  The proof of \cite[Thm.\ 20.1.34]{dcmca} applies verbatim: It is a
  standard reduction to the module case and the assumption on $M$
  ensures that Theorem~\ref{thm:main} applies.
\end{proof}

\begin{disc}
  If the $R$-complex $M$ in Corollary~\ref{cor20260929a} has bounded
  homology, then the assumption on the homology modules implies that
  $M$ has finite projective dimension; see for example Aihara et.\ al
  \cite[Thm.\ 4.1]{AAITY-14}
\end{disc}

This project arose from the question, does the Dimension Inequality
hold over fiber product rings?  As the next example shows, the answer
is affirmative for a particular class of fiber products. These fiber
product rings are precisely the local rings with decomposable maximal
ideal; see Ogoma~\cite[Lem.\ 3.1]{ogoma:edc}.

\begin{ex}\label{para20260927s}
  Let $(S,\mathfrak{s},k)$ and $(T,\mathfrak{t},k)$ be commutative
  noetherian local rings such that $S\neq k\neq T$ and let
  $R=S\times_k T$ be the fiber product of $S$ and $T$ over $k$, i.e.
  \begin{equation*}
    R=\left\{(s,t)\in S\times T\mid \pi_S(s)=\pi_T(t)\right\} 
  \end{equation*}
  where $S\xra{\pi_S} k\xla{\pi_T}T$ are the natural surjections. It
  is a commutative noetherian local ring with
  $\fm=\mathfrak{s}\oplus \mathfrak{t}$ and residue field $k$. By work
  of Lescot~\cite{lescot:sbpfal} one has
  \begin{equation*}
    \depth R=\min\{\depth S,\depth T,1\} \:.
  \end{equation*}
  In particular, $\depth R\leq 1$ holds and, therefore,
  Theorem~\ref{thm:main} applies.
\end{ex}

\begin{disc}\label{disc20260929a}
  The ring considered in \cite[Exa.~18.5.14]{dcmca}---already
  mentioned in Remark~\ref{disc20260929b}---is the fiber product
  $k[\![x]\!]\times_k k[\![y]\!]$. Thus, the assumption on the
  projective dimension of $M$ can not be dropped for fiber products.

  The local ring in Ma's~\cite[\S2.3]{Ma} counterexample to the
  Dimension Inequality is also a fiber product but, of course,
  different from those considered above: Ma's ring $A\times_CB$ is
  glued over a ring $C$ of dimension $2$ rather than over the residue
  field.
\end{disc}

Here is a local ring that is covered by Theorem~\ref{thm:main} and not
a fiber product:

\begin{ex}
  The ring $R=k[\![x,y]\!]/(y^2-x^3)$ is a Cohen--Macaulay local
  domain (hence not a fiber product) with $\dim R=1$, so
  Theorem~\ref{thm:main} applies to $R$. Set $M = R/(x)$ and
  $N = \mathfrak{m}$. As $x$ is not a zerodivisor, one has
  $\pd_R{M} = 1$ and $\dim_R M = 0 = \dim_R(M\otimes_RN)$. Since $R$ is
  a domain, $\dim_RN = \dim R = 1$ holds.  Thus equality holds in the
  Dimension Inequality $(1)$ as well as in the inequality $(2)$.
\end{ex}

In closing we note that the proof of \cite[Prop.\ B.5(1)]{Ma} adapts
to show that also the general version of the Dimension Inequality
holds for rings of dimension at most~$2$.

\begin{prop}\label{prop20260930a}
  Let $R$ be a commutative noetherian local ring and $M$ and $N$ be
  finitely generated $R$-modules. If $\pd_R M$ is finite and
  $\dim R\leq 2$ holds, then one has
  \begin{equation*}
    \dim_R M + \dim_R N \leq \dim R + \dim_R(M \otimes_R N)\:.
  \end{equation*}
\end{prop}

\begin{proof}
  First note that by Theorem~\ref{thm:main} we can assume that
  $\dim R=2$ holds.

  If $\dim_R M=0$ holds, then the asserted inequality holds trivially.
  If $\dim_R M=1$ holds, then there is, as in the proof
  of~\cite[Prop.\ B.5]{Ma}, a non-zero-divisor $z\in R$ with
  $V(z)\subseteq\supp_RM$ and, therefore,
  \begin{equation*}
    \supp_R(N/zN)=\supp_RN\cap
    V(z)\subseteq\supp_RN\cap\supp_RM=\supp_R(M\otimes_RN) \:.
  \end{equation*}
  Thus, $\dim_R(N/zN)\leq\dim_R(M\otimes_RN)$ holds, whence one has
  \begin{equation*}
    \dim_R N\leq 1+\dim_R(N/zN)\leq 1+\dim_R(M\otimes_RN)\:.
  \end{equation*}
  If $\dim_R M=2$ holds, then $\supp_RM=\spec R$, see again the proof
  of~\cite[Prop.\ B.5]{Ma}, and thus
  $\supp_R(M\otimes_RN)=\supp_RN$ holds. In particular,
  one has $\dim_R{N} = \dim_R(M\otimes_RN)$, and the asserted
  inequality follows; actually equality holds.
\end{proof}


\subsection*{AI disclosure}
This note arose from a conversation between the second author and
Claude (Anthropic), continuing work the authors had previously
started. Theorem~\ref{thm:main} was developed in the dialogue, with errors
caught directly or via Macaulay2 analysis by the authors. All
statements and proofs were written by the authors.

\providecommand{\bysame}{\leavevmode\hbox to3em{\hrulefill}\thinspace}
\providecommand{\MR}{\relax\ifhmode\unskip\space\fi MR }
\providecommand{\MRhref}[2]{%
  \href{http://www.ams.org/mathscinet-getitem?mr=#1}{#2} }
\providecommand{\href}[2]{#2}

\end{document}